\documentclass{amsart}
\pdfoutput=1

\allowdisplaybreaks
\usepackage{graphicx,epstopdf}
\DeclareGraphicsExtensions{.ps}
\usepackage{amsmath,amssymb}

\usepackage{MnSymbol}

\usepackage{enumerate}

\usepackage{array}
\newcolumntype{L}[1]{>{\raggedright\let\newline\\\arraybackslash\hspace{0pt}}m{#1}}
\newcolumntype{C}[1]{>{\centering\let\newline\\\arraybackslash\hspace{0pt}}m{#1}}
\newcolumntype{R}[1]{>{\raggedleft\let\newline\\\arraybackslash\hspace{0pt}}m{#1}}

\usepackage{mathtools}

\usepackage{pdflscape}

\usepackage{multirow}

\usepackage[refpage,noprefix]{nomencl}
\usepackage{multicol}
\makenomenclature

\usepackage{hyperref}
\hypersetup{colorlinks=true,citecolor=blue,linkcolor=blue}
\usepackage[capitalize]{cleveref}
\crefformat{equation}{(#2#1#3)}

\usepackage{tikz}
\usetikzlibrary{decorations.markings}
\usetikzlibrary{arrows}
\usepackage{tkz-berge}
\usepackage{tkz-euclide}
\usetkzobj{all} 

\usepackage{color}
\definecolor{darkgreen}{rgb}{0,0.7,0}

\numberwithin{equation}{section}
\newtheorem{theorem}{Theorem}[section]
\newtheorem{proposition}[theorem]{Proposition}
\newtheorem{lemma}[theorem]{Lemma}
\newtheorem{corollary}[theorem]{Corollary}

\theoremstyle{definition}

\newtheorem{remark}[theorem]{Remark}
\newtheorem{example}[theorem]{Example}

\DeclareMathOperator{\Span}{Span}

\DeclareMathOperator{\Proj}{Proj}

\newcommand{\set}[1]{{\left\lbrace #1 \right\rbrace}}
\newcommand{\br}[1]{{\left\langle #1 \right\rangle}}
\newcommand{\ck}{^{\vee}}
\newcommand{\integers}{\mathbb{Z}}

\newcommand{\aff}{\mathrm{aff}}
\newcommand{\fin}{\mathrm{fin}}
\newcommand{\re}{\mathrm{re}}
\newcommand{\reals}{\mathbb{R}}

\newcommand{\laff}{\triangleleft}
\newcommand{\gaff}{\triangleright}

\newcommand{\eigenspace}[1]{U^{#1}}
\newcommand{\RSChar}{\Phi}
\newcommand{\RS}{\RSChar}
\newcommand{\RSre}{\RS^\re}
\newcommand{\RSpos}{\RS^+}

\newcommand{\RSfin}{\RS_\fin}
\newcommand{\RSfinpos}{\RSfin^+}

\newcommand{\SimplesChar}{\Pi}
\newcommand{\Simples}{\SimplesChar}

\newcommand{\RSTChar}{\Upsilon}
\newcommand{\RST}[1]{\RSTChar^{#1}}
\newcommand{\RSTfin}[1]{\RST{#1}_\fin}
\newcommand{\SimplesTChar}{\Xi}
\newcommand{\SimplesT}[1]{\SimplesTChar^{#1}}

\newcommand{\TravInfChar}{\Psi}
\newcommand{\TravInf}[1]{\TravInfChar^{#1}}

\newcommand{\TravProj}[1]{\overrightarrow{\TravInfChar}^{#1}}

\newcommand{\TravInj}[1]{\overleftarrow{\TravInfChar}^{#1}}
\newcommand{\TravRegChar}{\Omega}
\newcommand{\TravReg}[1]{\TravRegChar^{#1}}
\newcommand{\AP}[1]{\RS_{#1}}

\newcommand{\dynkinradius}{.06cm}
\newcommand{\dynkinstep}{.5cm}
\newcommand{\dynkinlinesep}{.08cm}
\newcommand{\dynkinaffinedot}[4]{\fill[fill=red] (\dynkinstep*#1,\dynkinstep*#2) circle (\dynkinradius); \node at (\dynkinstep*#1,\dynkinstep*#2) [font=\tiny,#4] {$#3$};}
\newcommand{\dynkindot}[4]{\fill (\dynkinstep*#1,\dynkinstep*#2) circle (\dynkinradius); \node at (\dynkinstep*#1,\dynkinstep*#2) [font=\tiny,#4] {$#3$};}
\tikzset{
  line/.style={thin},
  dotline/.style={dotted},
  doubleline/.style={thin,double distance=\dynkinlinesep,postaction={decorate,decoration={markings,mark=at position 0.7 with {\arrow[line width=0.06cm]{angle 60}}}}},
  tripleline/.style={thin,double distance=\dynkinlinesep*1.5,postaction={decorate,decoration={markings,mark=at position 0.7 with {\arrow[line width=0.06cm]{angle 60}}}},postaction={draw}}
}

\newcommand{\dynkinline}[4]{\draw[line] (\dynkinstep*#1,\dynkinstep*#2) -- (\dynkinstep*#3,\dynkinstep*#4);}
\newcommand{\dynkindotline}[4]{\draw[dotline] (\dynkinstep*#1,\dynkinstep*#2) -- (\dynkinstep*#3,\dynkinstep*#4);}
\newcommand{\dynkindoubleline}[4]{\draw[doubleline] (\dynkinstep*#1,\dynkinstep*#2) -- (\dynkinstep*#3,\dynkinstep*#4);}
\newcommand{\dynkintripleline}[4]{\draw[tripleline] (\dynkinstep*#1,\dynkinstep*#2) -- (\dynkinstep*#3,\dynkinstep*#4);}
\newenvironment{dynkin}{\begin{tikzpicture}[baseline]}{\end{tikzpicture}}

\newcommand{\newword}[1]{\textbf{\emph{#1}}}

\usepackage[draft]{say}

\author{Nathan Reading}
\author{Salvatore Stella}
\title{The action of a Coxeter element on an affine root system}
\address[N. Reading]{Department of Mathematics, North Carolina State University, Raleigh, NC, USA}
\address[S. Stella]{Department of Mathematics \& Department of Computer Science, University of Haifa, Haifa, Mount Carmel 31905, Israel}
\thanks{Nathan Reading was supported in part by NSF grants DMS-1101568 and DMS-1500949.\\ \indent Salvatore Stella was partially supported by NCSU, INdAM, and the ISF grant 1144/16.}

\begin{document}

\begin{abstract}
The characterization of orbits of roots under the action of a Coxeter element is a fundamental tool in the study of finite root systems and their reflection groups.
This paper develops the analogous tool in the affine setting, adding detail and uniformity to a result of Dlab and Ringel.
\end{abstract}

\maketitle

\vspace{-8pt}

\setcounter{tocdepth}{1}
\tableofcontents

\section{Introduction}\label{intro}
Among the most important early results in the study of finite reflection groups is the description of orbits of roots (or reflecting hyperplanes) under the action of a Coxeter element---the product of a permutation of the simple reflections.
This description was proved uniformly by Steinberg~\cite{SteinbergFRG}, who analyzed the action of Coxeter elements on the Coxeter plane---a certain plane first considered by Coxeter in \cite{CoxeterRP}.
(See also \cite{plane}.)
Steinberg's uniform construction of the Coxeter plane was based on a careful analysis of eigenvalues and eigenvectors which led to (and was motivated by) a uniform proof of the formula $nh$ for the number of roots, where $n$ is the rank of the root system and $h$ is its \newword{Coxeter number} (the order of a Coxeter element).

We quote a version of this result that gives a \newword{transversal} of the orbits (a collection consisting of exactly one element from each of the orbits).
Let $\Phi$ be a root system of rank $n$, and let $c$ be any Coxeter element in the associated Weyl group.
We label the simple roots so that $c=s_1\cdots s_n$ where $s_i$ is the reflection with respect to the simple root $\alpha_i$ of $\Phi$. 
Define
\begin{align}
\label{rep->}
\TravProj{c}&:=\set{\alpha_1,s_1\alpha_2,\ldots,s_1\cdots s_{n-1}\alpha_n},\\
\label{rep<-}
\TravInj{c}&:=\set{\alpha_n,s_n\alpha_{n-1},\ldots,s_n\cdots s_2\alpha_1},
\end{align}%
\nomenclature[zzyzzzzc0]{$\TravInf{c}$}{$\TravProj{c}\cup\TravInj{c}$, transversal of infinite $c$-orbits in $\RS$}%
and write $\TravInf{c}$ for the union of the two sets.
The following is \cite[Proposition~VI.1.33]{bourbaki}.
\begin{theorem}\label{fin c-orbits}
Suppose $\RS$ is an irreducible finite root system and $c$ is a Coxeter element in the associated Weyl group $W$.
There are exactly $n$ $c$-orbits in $\RS$. 
The set $\TravProj{c}$ is a transversal of these orbits.
The set $\TravInj{c}$ is also a transversal of these orbits.
Each orbit has cardinality equal to the Coxeter number $h$.
\end{theorem}

Dlab and Ringel~\cite{Dlab76} proved an analogous result for the action of a Coxeter element~$c$ on an affine root system $\RS$.
Here, we improve on~\cite{Dlab76} by arguing uniformly (rather than type-by-type in the classification of affine root systems) and by clarifying some details for finite orbits.
Our immediate motivation is to support an almost-positive roots model \cite{affdenom} for cluster algebras of affine type.
(The finite-type model \cite{associahedra,FoZe03a,FoZe03,MRZ,Ste13} 
uses \cref{fin c-orbits} by way of \cite[Exercise~V\S6.2]{bourbaki}.)

The arguments in~\cite{Dlab76} for infinite orbits are easy and uniform, and ours are the same.  
The difficulty lies in the treatment of finite orbits.
Just as Steinberg's uniform proof of \cref{fin c-orbits} rests on an analysis of eigenvalues and eigenvectors, our analysis of finite orbits relies on a characterization, in \cref{eigen}, of the eigenvalues and eigenvectors of a Coxeter element in a Weyl group of affine type.
(Previous work on the affine eigenvalue problem, with different emphases, includes~\cite{AC,BLM,Coleman,Howlett,SteinbergEigen}.)
We will see that~$c$ has a linearly independent set of $(n-1)$ eigenvectors, spanning a hyperplane that we call~$\eigenspace{c}$.  We define $\RST{c}$ to be $\RS\cap \eigenspace{c}$\nomenclature[zzuc]{$\RST{c}$}{$\RS\cap \eigenspace{c}$}. 
Writing $\RSfin$ for the finite root system associated to $\RS$, we write $\RSTfin{c}$ for $\RSfin\cap \eigenspace{c}$.\nomenclature[zzufinc]{$\RSTfin{c}$}{$\RSfin\cap \eigenspace{c}$}
We will see in \cref{Up details} that $\RSTfin{c}$ is a root system of rank $n-2$.  
There is canonical system $\SimplesT{c}_\fin$ \nomenclature[zzoc]{$\SimplesT{c}$}{simple roots for $\RSTfin{c}$} of simple roots for $\RSTfin{c}$ (the unique simple system for $\RSTfin{c}$ that is a subset of~$\RSpos$).  

Let $\delta$ be the positive imaginary root closest to the origin.
Let $\aff\in\set{1,\ldots,n}$ be the index such that $\alpha_\aff$ is the unique simple root of $\RS$ that is not in $\RSfin$.
To keep track of the location of the letter $s_\aff$ in the expression $s_1\cdots s_n$ for $c$, we let $c_\laff=s_1\cdots s_{\aff-1}$ and let $c_\gaff=s_{\aff+1}\cdots s_n$, so that $c=c_\laff s_\aff c_\gaff$.
We will see in \cref{Up details} that the roots in $\SimplesT{c}_\fin$ can be ordered $\beta_1,\dots,\beta_{n-2}$ so that $t_{\beta_1}\cdots t_{\beta_{n-2}}=c_\laff t_\theta c_\gaff $, where $\theta$ is the highest root or highest short root in $\RSfinpos$ as explained in \cref{sec:root_systems} and $t_\alpha$ stands for the reflection orthogonal to a root~$\alpha$.
For such an ordering of $\SimplesT{c}_\fin$, define $\TravReg{c}:=\set{\beta_1,t_{\beta_1}\beta_2,\ldots,t_{\beta_1}\cdots t_{\beta_{n-3}}\beta_{n-2}}$.  \nomenclature[zzzc1]{$\TravReg{c}$}{transversal of the finite $c$-orbits in $\AP{c}$}
For any $\beta\in\TravReg{c}$, let $\kappa(\beta)$ be the smallest positive integer such that $\kappa(\beta)\delta-\beta$ is a root.
(The existence of $\kappa(\beta)$ will be established as part of the proof of the main theorem, where we will also see that $\kappa(\beta)$ is constant on components of $\RSTfin{c}$.)

We now state the result on $c$-orbits in affine type.  
\begin{theorem}[Cf. Dlab-Ringel \protect{\cite[Chapter~1]{Dlab76}}]\label{aff c-orbits}
  Suppose $\RS$ is an affine root system and $c$ is a Coxeter element in the associated Weyl group $W$.   
\begin{enumerate}[\qquad \upshape (1)]
    \item	\label{infinite transversal}
      There are exactly $2n$ infinite $c$-orbits in $\RS$. 
      The set $\TravInf{c}$ is a transversal of these orbits.
    \item \label{criterion}
      The $c$-orbit of a root $\beta\in\RS$ is finite if and only if $\beta\in \eigenspace{c}$.
    \item \label{im orb}
      Every imaginary root is fixed by $c$.
    \item \label{finite transversal}
      For $\RS$ of rank $2$, there are no finite $c$-orbits of real roots.
      For larger rank, there are infinitely many finite $c$-orbits of real roots and the set $\set{\beta + m\cdot\kappa(\beta)\delta:\beta\in\TravReg{c},\,m\in\integers}$ is a transversal of them.
    \item	\label{pos neg}
      Each finite $c$-orbit contains either only positive roots or only negative roots. 
      In particular, the $c$-orbit of a real root $\beta + m\cdot\kappa(\beta)\delta$ for $\beta\in\TravReg{c}$ consists of positive roots if and only if $m\ge0$.
    \item \label{finite transversal finite}
      A finite $c$-orbit intersects $\RSfinpos$ if and only if it intersects $\TravReg{c}$.
  \end{enumerate}		
\end{theorem}

\begin{example}\label{ex:orb}  
Consider the Cartan matrix 
$A =\begin{bsmallmatrix*}[r]
        2 & -2 & 0 \\
        -1 & 2 & -1 \\
        0 & -2 & 2
\end{bsmallmatrix*}$
corresponding to the Dynkin diagram
\raisebox{3pt}{    \begin{dynkin}
    \dynkindoubleline{1}{0}{-0.25}{0}
    \dynkindoubleline{1}{0}{2.25}{0}
      \dynkindot{-0.25}{0}{1}{below}
      \dynkindot{1}{0}{2}{below}
      \dynkinaffinedot{2.25}{0}{3}{below}
    \end{dynkin}}
of type $D_3^{(2)}$.
(The associated root system is not a standard affine root system in the sense of \cref{ex tab}.)
  Let $c=s_1s_2s_3$.
In the basis of simple roots $\Simples=\{\alpha_1,\alpha_2,\alpha_3\}$, $c$ is given by the matrix
  $
 \begin{bsmallmatrix*}[r]
        1 & \,\,2 & -2 \\
        1 & 1 & -1 \\
        0 & 2 & -1
\end{bsmallmatrix*}
 $.
Its eigenvectors are $\alpha_1+\alpha_3$ with eigenvalue~$-1$, and $\delta=\alpha_1+\alpha_2+\alpha_3$ with eigenvalue~$1$.
(Looking forward to \cref{eigen}, the generalized $1$-eigenvector $\gamma_c$ is $\alpha_1+\frac{1}{2}\alpha_2$.)
  
There are $6$ infinite $c$-orbits and a transversal of them is $\TravInf{c}=\TravInj{c}\cup\TravProj{c}$ with
  \begin{eqnarray*}
    \TravInj{c}=
    \{
      s_3s_2\alpha_1 = \alpha_1+\alpha_2+2\alpha_3,
      \quad
      s_3\alpha_2 = \alpha_2+2\alpha_3,
      \quad
      \alpha_3
    \}\hphantom{.}
    \\
    \TravProj{c}=
    \{
      \alpha_1,
      \quad
      s_1\alpha_2 = 2\alpha_1+\alpha_2,
      \quad
      s_1s_2\alpha_3 = 2\alpha_1+\alpha_2+\alpha_3
    \}.
\end{eqnarray*}
We have $\TravReg{c}=\set{\alpha_2}$, and the $c$-orbit of $\alpha_2$ is $\set{\alpha_2,2\alpha_1+\alpha_2+2\alpha_3}$.
Since $\delta-\alpha_2=\alpha_1+\alpha_3$ is not a root but $2\delta-\alpha_2=2\alpha_1+\alpha_2+2\alpha_3$ is a root, we have $\kappa(\alpha_2)=2$.
The finite orbits are the $2\delta$-translates of this orbit.

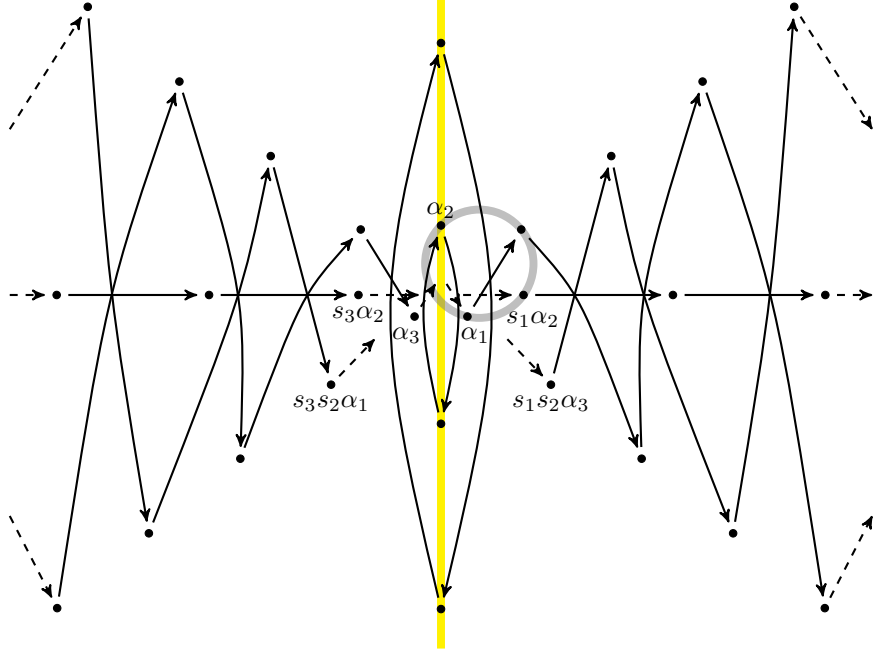
\begin{figure}
  \begin{tikzpicture}[scale=0.147]
    \pgfdeclarelayer{bg}
    \pgfsetlayers{bg,main}

    \node[draw,circle,inner sep=1pt,fill,outer sep=3pt, label={[label distance=-5pt] $\alpha_2$}] (010) at (0.0, 6.2925287399) {};
    \node[draw,circle,inner sep=1pt,fill,outer sep=3pt] (111) at (0.0, -11.6568542495) {};
    \draw [thick, ->, >=stealth'] (010) .. controls (2,0) and (2,-3) ..  (111);
    \draw [thick, ->, >=stealth'] (111) .. controls (-2,-3) and (-2,0) ..  (010);

    \node[draw,circle,inner sep=1pt,fill,outer sep=3pt] (131) at (0.0, 22.8028337921) {};
    \node[draw,circle,inner sep=1pt,fill,outer sep=3pt] (232) at (0.0, -28.4249924808) {};
    \draw [thick, ->, >=stealth'] (131) .. controls (6,0) and (6,-3) ..  (232);
    \draw [thick, ->, >=stealth'] (232) .. controls (-6,-3) and (-6,0) ..  (131);

    \node[draw,circle,inner sep=1pt,fill,outer sep=3pt] (384) at (-31.9448643066, 26.0828724845) {};
    \node[draw,circle,inner sep=1pt,fill,outer sep=3pt] (364) at (-26.4180193581, -21.5702224808) {};
    \node[draw,circle,inner sep=1pt,fill,outer sep=3pt] (142) at (-15.3979123914, 12.572342821) {};
    \node[draw,circle,inner sep=1pt,fill,outer sep=3pt, label={[label distance=-5pt, yshift=-1.3em] $s_3s_2\alpha_1$}] (122) at (-9.9409110857, -8.1167199128) {};
    \draw [thick, ->, >=stealth', dashed] (-39,15) -- (384);
    \draw [thick, ->, >=stealth'] (384) .. controls (-30,0) ..  (364);
    \draw [thick, ->, >=stealth'] (364) .. controls (-18,0) ..  (142);
    \draw [thick, ->, >=stealth'] (142) -- (122);
    \draw [thick, ->, >=stealth', dashed] (122) -- (-6, -4);

    \node[draw,circle,inner sep=1pt,fill,outer sep=3pt, label={[label distance=-5pt, xshift=3pt, yshift=-1.3em] $\alpha_1$}] (100) at (2.3898765979, -1.951326071) {};
    \node[draw,circle,inner sep=1pt,fill,outer sep=3pt] (120) at (7.2588348073, 5.9268138017) {};
    \node[draw,circle,inner sep=1pt,fill,outer sep=3pt] (342) at (18.1455919264, -14.8158137668) {};
    \node[draw,circle,inner sep=1pt,fill,outer sep=3pt] (362) at (23.6573394154, 19.3161367465) {};
    \node[draw,circle,inner sep=1pt,fill,outer sep=3pt] (584) at (34.7101601711, -28.3407271032) {};
    \draw [thick, ->, >=stealth', dashed] (0.5, 1) -- (100);
    \draw [thick, ->, >=stealth'] (100) -- (120);
    \draw [thick, ->, >=stealth'] (120) .. controls (12.6,0) .. (342);
    \draw [thick, ->, >=stealth'] (342) .. controls (17.7,0) .. (362);
    \draw [thick, ->, >=stealth'] (362) .. controls (30.3,0) .. (584);
    \draw [thick, ->, >=stealth', dashed] (584) -- (39,-20);

    \node[draw,circle,inner sep=1pt,fill,outer sep=3pt] (253) at (-34.7561038499, 0.0) {};
    \node[draw,circle,inner sep=1pt,fill,outer sep=3pt] (132) at (-20.9753100897, 0.0) {};
    \node[draw,circle,inner sep=1pt,fill,outer sep=3pt, label={[label distance=-5pt, yshift=-1.3em] $s_3\alpha_2$}] (011) at (-7.4641016151, 0.0) {};
    \draw [thick, ->, >=stealth', dashed] (-39,0) -- (253);
    \draw [thick, ->, >=stealth'] (253) -- (132);
    \draw [thick, ->, >=stealth'] (132) -- (011);
    \draw [thick, ->, >=stealth', dashed] (011) -- (-2, 0);
    
    \node[draw,circle,inner sep=1pt,fill,outer sep=3pt, label={[label distance=-5pt, xshift=0.35em, yshift=-1.5em] $s_1\alpha_2$}] (110) at (7.4641016151, 0.0) {};
    \node[draw,circle,inner sep=1pt,fill,outer sep=3pt] (231) at (20.9753100897, 0.0) {};
    \node[draw,circle,inner sep=1pt,fill,outer sep=3pt] (352) at (34.7561038499, 0.0) {};
    \draw [thick, ->, >=stealth', dashed] (2, 0) -- (110);
    \draw [thick, ->, >=stealth'] (110) -- (231);
    \draw [thick, ->, >=stealth'] (231) -- (352);
    \draw [thick, ->, >=stealth', dashed] (352) -- (39,0);

    \node[draw,circle,inner sep=1pt,fill,outer sep=3pt] (485) at (-34.7101601711, -28.3407271032) {};
    \node[draw,circle,inner sep=1pt,fill,outer sep=3pt] (263) at (-23.6573394154, 19.3161367465) {};
    \node[draw,circle,inner sep=1pt,fill,outer sep=3pt] (243) at (-18.1455919264, -14.8158137668) {};
    \node[draw,circle,inner sep=1pt,fill,outer sep=3pt] (021) at (-7.2588348073, 5.9268138017) {};
    \node[draw,circle,inner sep=1pt,fill,outer sep=3pt, label={[label distance=-5pt, xshift=-3pt, yshift=-1.3em] $\alpha_3$}] (001) at (-2.3898765979, -1.951326071) {};
    \draw [thick, ->, >=stealth', dashed] (-39,-20) -- (485);
    \draw [thick, ->, >=stealth'] (485) .. controls (-30.3,0) ..  (263);
    \draw [thick, ->, >=stealth'] (263) .. controls (-17.7,0) ..  (243);
    \draw [thick, ->, >=stealth'] (243) .. controls (-12.6,0) ..  (021);
    \draw [thick, ->, >=stealth'] (021) -- (001);
    \draw [thick, ->, >=stealth', dashed] (001) -- (-0.5, 1);

    \node[draw,circle,inner sep=1pt,fill,outer sep=3pt, label={[label distance=-5pt, yshift=-1.3em] $s_1s_2\alpha_3$}] (221) at (9.9409110857, -8.1167199128) {};
    \node[draw,circle,inner sep=1pt,fill,outer sep=3pt] (241) at (15.3979123914, 12.572342821) {};
    \node[draw,circle,inner sep=1pt,fill,outer sep=3pt] (463) at (26.4180193581, -21.5702224808) {};
    \node[draw,circle,inner sep=1pt,fill,outer sep=3pt] (483) at (31.9448643066, 26.0828724845) {};
    \draw [thick, ->, >=stealth', dashed] (6,-4) -- (221);
    \draw [thick, ->, >=stealth'] (221) -- (241);
    \draw [thick, ->, >=stealth'] (241) .. controls (18,0) .. (463);
    \draw [thick, ->, >=stealth'] (463) .. controls (30,0) .. (483);
    \draw [thick, ->, >=stealth', dashed] (483) -- (39,15) ;

    \begin{pgfonlayer}{bg}
      \draw[color=yellow, line width=3pt] (0, 27) -- (0, -32);

      \tkzCircumCenter(100,010,120)\tkzGetPoint{O}
      \tkzDrawCircle[color=gray, line width=3pt, opacity=0.5](O,100)
    \end{pgfonlayer}

  \end{tikzpicture}
\caption{The action of $c$ on positive roots, for $\RS$ and $c$ as in \cref{ex:orb}}
\label{fig:orb}
\end{figure}  
We depict parts of some orbits in \cref{fig:orb}. 
The figure is obtained by drawing a ray from the origin through each root, taking the intersection with a unit sphere centered at $0$ and stereographically projecting the result from the direction~$\delta$. 
We plot only positive roots of small height.
Negative roots would be concentrated in the center of the picture, while positive roots of greater height would be further out towards the edge of the figure. 
The imaginary root $\delta$ is  the point at infinity.
The yellow line indicates $\eigenspace{c}=\Span\{\alpha_1+\alpha_3,\delta\}$, and the gray circle indicates $\Span\RSfin$.
Arrows give the action of $c$ on roots.
Arrows to or from roots that do not appear in the figure are shown dotted.
\end{example}

\begin{example}\label{ex tab}
\begin{table}[bp]
\medskip
  \centering
  \begin{tabular}{|c|c|c|l|}
    
    \hline
    
    \hspace{-0.25ex}Type\hspace{-0.25ex} & Diagram of $\RS$ & \hspace{-0.55ex}Diagram of $\RSTfin{c}$\hspace{-0.55ex} & \multicolumn{1}{c|}{Simple roots of $\RSTfin{c}$}\\
    
    \hline&&&\\[-11pt]
		\hline&&& \\[-9pt]
		  
      $A^{(1)}_1$ 
      
      &	
      
      \raisebox{4pt}{\begin{dynkin}
        \draw[thin,double distance=\dynkinlinesep,postaction={decorate,decoration={markings,mark=at position 0.15 with {\arrowreversed[line width=0.06cm,xshift=-1pt]{angle 60}},mark=at position 0.85 with {\arrow[line width=0.06cm,xshift=1pt]{angle 60}}}}](0,0) -- (1.6*\dynkinstep,0);
			  \dynkindot{0}{0}{1}{below}
			  \dynkinaffinedot{1.6}{0}{2}{below}
		  \end{dynkin}}
		  
      &

      &
      
      \\
    
    \hline

      $\begin{array}{c}A^{(1)}_{n-1}\\[-2pt] _{(n\ge3)}\\[-2pt] _{k\neq n}\end{array}$ 
      
      &
		
      \begin{dynkin}
  			\dynkinline{1}{0.7}{0.5}{.7}
  			\dynkinline{2}{0.7}{2.5}{.7}
  			\dynkinline{1}{-0.7}{0.5}{-0.7}
  			\dynkinline{2}{-0.7}{2.5}{-0.7}
  			\dynkinline{0}{0}{0.5}{.7}
  			\dynkinline{0}{0}{0.5}{-.7}
  			\dynkindotline{1}{.7}{2}{.7}
  			\dynkindotline{1}{-.7}{2}{-.7}
  			\dynkinline{2.5}{.7}{3}{0}
  			\dynkinline{2.5}{-.7}{3}{0}
  			\dynkindot{0}{0}{1}{left}
  			\dynkindot{0.5}{.7}{2}{above}
  			\dynkindot{0.5}{-.7}{k+1}{below}
  			\dynkindot{2.5}{.7}{k}{above}
  			\dynkindot{2.5}{-.7}{n-1}{below}
  			\dynkinaffinedot{3}{0}{n}{right}
  		\end{dynkin}

		  & 
			
      \begin{dynkin}
				\dynkinline{1}{0.7}{2.4}{0.7}
				\dynkinline{3.6}{0.7}{4}{0.7}
				\dynkindotline{2.4}{0.7}{3.6}{0.7}
				\dynkindot{1}{0.7}{1}{below}
				\dynkindot{2}{0.7}{2}{below}
				\dynkindot{4}{0.7}{k-1}{below}

				\dynkinline{1}{-0.3}{2.4}{-0.3}
				\dynkinline{3.6}{-0.3}{4}{-0.3}
				\dynkindotline{2.4}{-0.3}{3.6}{-0.3}
				\dynkindot{1}{-0.3}{k}{below}
				\dynkindot{2}{-0.3}{k+1}{below}
				\dynkindot{4}{-0.3}{n-2}{below}
			\end{dynkin}
			
      & 
			
      {\small$
      \begin{array}{l}
        \beta_j=\alpha_{j+1} 
      \end{array}
      $}

      \\

    \hline

      $\begin{array}{c}B^{(1)}_{n-1}\\[-2pt] _{(n\ge4)}\end{array}$ 
      
      &	

  		\begin{dynkin}
  			\dynkinline{0}{0}{0.5}{0}
 			\dynkinline{1.5}{0}{2}{0}
  			\dynkinline{2.5}{.7}{2}{0}
  			\dynkinline{2.5}{-.7}{2}{0}
  			\dynkindotline{0.5}{0}{1.5}{0}
  			\dynkindoubleline{0}{0}{-1.25}{0}
  			\dynkindot{2.5}{.7}{n-1}{right}
  			\dynkinaffinedot{2.5}{-.7}{n}{right}
  			\dynkindot{2}{0}{n-2}{right}
  			\dynkindot{0}{0}{2}{below}
  			\dynkindot{-1.25}{0}{1}{below}
  		\end{dynkin}
  		
      & 

			\begin{dynkin}
				\dynkinline{1}{0.7}{2.4}{0.7}
				\dynkinline{3.6}{0.7}{4}{0.7}
				\dynkindotline{2.4}{0.7}{3.6}{0.7}
				\dynkindot{1}{0.7}{1}{below}
				\dynkindot{2}{0.7}{2}{below}
				\dynkindot{4}{0.7}{n-3}{below}
				\dynkindot{2.5}{-0.3}{n-2}{below}
			\end{dynkin}
			
      & 
			
      {\small$
				\begin{array}{l}
          \beta_{n-2}=\sum_{i=1}^{n-1} \alpha_i\\
					\beta_j=\alpha_{j+1}
				\end{array}
			$}

      \\

    \hline

      $\begin{array}{c}C^{(1)}_{n-1}\\[-2pt] _{(n\ge3)}\end{array}$
        
      &

  		\begin{dynkin}
 			\dynkinline{1}{0}{1.5}{0}
			\dynkinline{2.5}{0}{3}{0}
  			\dynkindoubleline{-0.25}{0}{1}{0}
  			\dynkindotline{1.5}{0}{2.5}{0}
  			\dynkindoubleline{4.25}{0}{3}{0}
  			\dynkindot{-0.25}{0}{1}{below}
  			\dynkindot{1}{0}{2}{below}
  			\dynkindot{3}{0}{n-1}{below}
  			\dynkinaffinedot{4.25}{0}{n}{below}
  		\end{dynkin}

  		&

			\begin{dynkin}
				\dynkinline{1}{0.2}{2.4}{0.2}
				\dynkinline{3.6}{0.2}{4}{0.2}
				\dynkindotline{2.4}{0.2}{3.6}{0.2}
				\dynkindot{1}{0.2}{1}{below}
				\dynkindot{2}{0.2}{2}{below}
				\dynkindot{4}{0.2}{n-2}{below}
			\end{dynkin}
			
      & 
			
      {\small$
      \begin{array}{l}
        \beta_j=\alpha_{j+1}
      \end{array}
      $ 	}		

      \\

    \hline&&&\\[-9pt]

      $\begin{array}{c}D^{(1)}_{n-1}\\[-2pt] _{(n\ge5)}\end{array}$ 
      
      &
		
      \begin{dynkin}
			\dynkinline{0.5}{0}{1}{0}
			\dynkinline{2.5}{0}{2}{0}
  			\dynkinline{0}{.7}{0.5}{0}
  			\dynkinline{0}{-.7}{0.5}{0}
  			\dynkindotline{1}{0}{2}{0}
  			\dynkinline{2.5}{0}{3}{.7}
  			\dynkinline{2.5}{0}{3}{-.7}
  			\dynkindot{0}{.7}{1}{left}
  			\dynkindot{0}{-.7}{2}{left}
  			\dynkindot{0.5}{0}{3}{left}
  			\dynkindot{2.5}{0}{n-2}{right}
  			\dynkindot{3}{.7}{n-1}{right}
  			\dynkinaffinedot{3}{-.7}{n}{right}
  		\end{dynkin}

      &

			\begin{dynkin}
				\dynkinline{1}{1.3}{2.4}{1.3}
				\dynkinline{3.6}{1.3}{4}{1.3}
				\dynkindotline{2.4}{1.3}{3.4}{1.3}
				\dynkindot{1}{1.3}{1}{below}
				\dynkindot{2}{1.3}{2}{below}
				\dynkindot{4}{1.3}{n-4}{below}
				\dynkindot{2.5}{0.3}{n-3}{below}
				\dynkindot{2.5}{-0.7}{n-2}{below}
			\end{dynkin}
			
      & 
			
      {\small$
				\begin{array}{l}
          \beta_{n-3}=\alpha_1+\sum_{i=3}^{n-1} \alpha_i \\
					\beta_{n-2}=\alpha_2+\sum_{i=3}^{n-1} \alpha_i \\
					\beta_j=\alpha_{j+2}
				\end{array}
			$}

      \\[18pt]

    \hline

  		$E^{(1)}_6$ 
      
      &
		\raisebox{-12pt}{
  		\begin{dynkin}
  			\dynkinline{0}{0}{4}{0}
  			\dynkinline{2}{0}{2}{2}
  			\dynkindot{0}{0}{3}{below}
  			\dynkindot{1}{0}{4}{below}
  			\dynkindot{2}{0}{5}{below}
  			\dynkindot{3}{0}{6}{below}
  			\dynkinaffinedot{4}{0}{7}{below}
  			\dynkindot{2}{1}{2}{right}
  			\dynkindot{2}{2}{1}{right}
  		\end{dynkin}
		}
      & 

			\begin{dynkin}
				\dynkinline{1}{1.3}{2}{1.3}
				\dynkinline{1}{0.3}{2}{0.3}
				\dynkindot{1}{1.3}{1}{below}
				\dynkindot{2}{1.3}{2}{below}
				\dynkindot{1}{0.3}{3}{below}
				\dynkindot{2}{0.3}{4}{below}
				\dynkindot{1.5}{-0.7}{5}{below}
			\end{dynkin}

			& 
			
      {\small$
				\begin{array}{l}
					\beta_1= \alpha_{4} + \alpha_{5} \\
					\beta_2= \alpha_{1} + \alpha_{2} + \alpha_{5} + \alpha_{6} \\
					\beta_3= \alpha_{2} + \alpha_{5} \\
					\beta_4= \alpha_{3} + \alpha_{4} + \alpha_{5} + \alpha_{6} \\
					\beta_5= \alpha_{2} + \alpha_{4} + \alpha_{5} + \alpha_{6}
				\end{array}
			
			$}

		  \\

    \hline

	  	$E^{(1)}_7$
      
      &
		
      \begin{dynkin}
  			\dynkinline{0}{0}{6}{0}
  			\dynkinline{3}{0}{3}{1}
  			\dynkindot{0}{0}{2}{below}
  			\dynkindot{1}{0}{3}{below}
  			\dynkindot{2}{0}{4}{below}
  			\dynkindot{3}{0}{5}{below}
  			\dynkindot{4}{0}{6}{below}
  			\dynkindot{5}{0}{7}{below}
  			\dynkinaffinedot{6}{0}{8}{below}
  			\dynkindot{3}{1}{1}{right}
  		\end{dynkin}

      &

			\begin{dynkin}
	      \dynkinline{1}{1.3}{3}{1.3}
        \dynkindot{1}{1.3}{1}{below}
        \dynkindot{2}{1.3}{2}{below}
        \dynkindot{3}{1.3}{3}{below}
        \dynkindot{1.5}{0.3}{4}{below}
        \dynkindot{2.5}{0.3}{5}{below}
        \dynkinline{1.5}{0.3}{2.5}{0.3}
        \dynkindot{2}{-0.7}{6}{below}
			\end{dynkin}
			
      & 
			
      {\small$
				\begin{array}{l}
					\beta_1=\alpha_{4} + \alpha_{5}\\
					\beta_2=\alpha_{1} + \alpha_{5} + \alpha_{6}\\
          \beta_3=\sum_{i=2}^7 \alpha_{i} \\
          \beta_4=\sum_{i=3}^6 \alpha_{i} \\
          \beta_5=\alpha_{1} + \sum_{i=4}^7 \alpha_{i} \\
          \beta_6=\alpha_{1} + \alpha_{5} + \sum_{i=3}^7 \alpha_{i}
				\end{array}
			$}

		  \\
    
    \hline
		
      $E^{(1)}_8$
      
      &
		
      \!\!\!\begin{dynkin}
		  	\dynkinline{1}{0}{8}{0}
		  	\dynkinline{3}{0}{3}{1}
		  	\dynkindot{1}{0}{2}{below}
		  	\dynkindot{2}{0}{3}{below}
		  	\dynkindot{3}{0}{4}{below}
		  	\dynkindot{4}{0}{5}{below}
		  	\dynkindot{5}{0}{6}{below}
		  	\dynkindot{3}{1}{1}{right}
		  	\dynkindot{6}{0}{7}{below}
		  	\dynkindot{7}{0}{8}{below}
		  	\dynkinaffinedot{8}{0}{9}{below}
		  \end{dynkin}\!\!\!
		  
      &

			\begin{dynkin}
			  \dynkinline{1}{1.3}{4}{1.3}
        \dynkindot{1}{1.3}{1}{below}
        \dynkindot{2}{1.3}{2}{below}
        \dynkindot{3}{1.3}{3}{below}
        \dynkindot{4}{1.3}{4}{below}
        \dynkindot{2}{0.3}{5}{below}
        \dynkindot{3}{0.3}{6}{below}
        \dynkinline{2}{0.3}{3}{0.3}
        \dynkindot{2.5}{-0.7}{7}{below}
			\end{dynkin}
			
			& 

		{\small	$
				\begin{array}{l}
					\beta_1=\alpha_{3} + \alpha_{4} + \alpha_{5}\\
          \beta_2=\alpha_{1} + \sum_{i=4}^6 \alpha_{i} \\
          \beta_3=\sum_{i=2}^7 \alpha_{i} \\
          \beta_4=\alpha_{1} + \sum_{i=3}^8 \alpha_{i} \\
          \beta_5=\alpha_{1} + \alpha_{4} + \sum_{i=3}^7 \alpha_{i} \\
          \beta_6=\alpha_{4} + \alpha_{5} + \sum_{i=1}^8 \alpha_{i} \\
          \beta_7=\alpha_{4} + \sum_{i=3}^6 \alpha_{i} + \sum_{i=1}^8 \alpha_{i}\!\!\\
				\end{array}
			$}

		  \\
    
    \hline
		
      $F^{(1)}_4$
      
      &
		
      \begin{dynkin}
  			\dynkinline{4}{0}{2}{0}
  			\dynkindoubleline{2}{0}{0.75}{0}
  			\dynkinline{0.75}{0}{-0.250}{0}
  			\dynkinaffinedot{4}{0}{5}{below}
  			\dynkindot{3}{0}{4}{below}
  			\dynkindot{2}{0}{3}{below}
  			\dynkindot{0.75}{0}{2}{below}
  			\dynkindot{-0.25}{0}{1}{below}
  		\end{dynkin}
  		
      & 

			\begin{dynkin}
				\dynkinline{1}{0.7}{2}{0.7}
				\dynkindot{1}{0.7}{1}{below}
				\dynkindot{2}{0.7}{2}{below}
				\dynkindot{1.5}{-0.3}{3}{below}
			\end{dynkin}
			
      &
			
      {\small$
				\begin{array}{l}
					\beta_1=\alpha_2+\alpha_3 \\
					\beta_2=\alpha_1+\alpha_2+\alpha_3+\alpha_4 \\
					\beta_3=2\alpha_2+\alpha_3+\alpha_4
				\end{array}
			$}

      \\

    \hline

		  $G^{(1)}_2$
      
      &	
		
      \begin{dynkin}
  			\dynkinline{2}{0}{1}{0}
  			\dynkintripleline{1}{0}{-0.25}{0}
  			\dynkinaffinedot{2}{0}{3}{below}
  			\dynkindot{1}{0}{2}{below}
  			\dynkindot{-0.25}{0}{1}{below}
  		\end{dynkin}

  		& 

			\begin{dynkin}
				\dynkindot{1}{0.3}{1}{below}
			\end{dynkin}

			& 
			
      {\small$
      \begin{array}{l}
        \beta_1=\alpha_1+\alpha_2
      \end{array}
      $ }			

      \\

    \hline
  \end{tabular}
  
 \bigskip
  
\caption{Standard affine root systems and their finite orbits}
\label{tab:type-by-type}
\end{table}
\cref{tab:type-by-type} gives the data necessary to describe finite orbits for the standard affine root systems and a particular choice of Coxeter element.
(These are the root systems that are obtained from finite root systems by the usual construction.  See \cref{sec:root_systems}.)
We name the types as in \cite[\S4.8]{Kac90}.
The choice of Coxeter element $c=s_1\cdots s_n$ is given by the labeling of nodes in the second column. 
In every case $\alpha_\aff$ is $\alpha_n$ so that $c=c_\laff s_\aff$, and we draw the affine node in a different color.
We also label the nodes in the diagram of $\RSTfin{c}$ according to \cref{Up details}, i.e.\ so that $c_\laff t_\theta=s_1\cdots s_{n-1}t_\theta=t_{\beta_1}\cdots t_{\beta_{n-2}}$.  
In each case, the integer $\kappa(\beta)$ is $1$ for all $\beta\in\TravReg{c}$. 
\end{example}

\begin{remark}
Details about $c$-orbits and $\RSTfin{c}$ in nonstandard affine root systems can be obtained from the information in \cref{tab:type-by-type} by rescaling.
See \cref{sec:root_systems}.
Furthermore, every Coxeter element in an affine Weyl group can be obtained from a Coxeter element described in \cref{tab:type-by-type} via source-sink moves (\cref{aff conj}).
Each source-sink move conjugates the Coxeter element by a simple reflection.
Thus, complete details on orbits of any affine root system under the action of any Coxeter element can be recovered from the details in \cref{tab:type-by-type} by conjugation and rescaling.
\end{remark}

\begin{remark}
The sets $\TravProj{c}$ and $\TravInj{c}$ depend only on~$c$, not on a choice of reduced word for $c$.
Each contains $n$ distinct positive roots.
\cref{aff c-orbits}, $\TravInf{c}$ implies that $\TravProj{c}$ and $\TravInj{c}$ are disjoint when $\RS$ is affine.
\end{remark}



%

\begin{remark}\label{tUp remark}
The set $\RST{c}$ is, in a certain sense, a ``root subsystem'' of $\RS$, but illustrates the strange rank behavior that can occur when one considers the intersections of infinite root systems with subspaces.
(See \cite[Remark~2.13]{typefree}.)
The set $\RST{c}$ can be obtained from $\RSTfin{c}$ by following the usual affinization procedure that adjoins~$\delta$.
However, this procedure, when applied to reducible finite root systems, gives the same result as extending each irreducible component using a new vector $\delta_i$ and then identifying all of the $\delta_i$ with~$\delta$.
The difference in rank between the original root system and the extended root system is (in effect) the number of irreducible components, even though only one extra dimension has been added.
Looking backward to \cref{tab:type-by-type}, we see that the rank of $\RST{c}$ can vary from $n-2$ to $n+1$ because $\RSTfin{c}$ has between $0$ and $3$ connected components.
\end{remark}

\begin{remark}
We will see in \cref{c on tUp components} that the sizes of finite orbits of real roots are the ranks of irreducible components of $\RST{c}$, which are one greater than the ranks of components of $\RSTfin{c}$.
\cref{tab:type-by-type} reveals that $\RSTfin{c}$ has at most $3$ irreducible components, so there are at most three different sizes of finite orbits of real roots.
\end{remark}

\begin{remark}\label{DR and us}
The difference between \cref{aff c-orbits} and the Dlab-Ringel result (as it can be pieced together from \cite[Proposition~1.9]{Dlab76} and the preceding and following remarks) is in \cref{aff c-orbits}\eqref{finite transversal}.
In place of $\kappa(\beta)$, Dlab-Ringel uses the \newword{tier number}~$r$ (the parenthesized superscript in the notation of \cite[Chapter~4]{Kac90}).
In every case, $\kappa(\beta)\le r$, but this inequality may be strict. 
When $\kappa(\beta)<r$, the Dlab-Ringel formulation must replace $\TravReg{c}$ by a larger set.
When $r=1$ (the standard types discussed in \cref{ex tab}), we must have $\kappa(\beta)=r=1$.
One can check that $\kappa(\beta)=r$ in types $D_n^{(2)}$ and $D_4^{(3)}$ and that $\kappa(\beta)=1<r$ in type $A_{2n-2}^{(2)}$.
The remaining types are more complicated:  
In type $A_{2n-3}^{(2)}$, we have $\kappa(\beta)=2$ for $\beta$ in the rank-$1$ component of $\RSTfin{c}$ and $\kappa(\beta)=1$ in the other component.
In type $E_6^{(2)}$, $\beta(\kappa)=1$ in the rank-$1$ component and $\beta(\kappa)=2$ in the rank-$2$ component.
\end{remark}

\begin{remark}
Several of the ideas of this paper can also be found in \cite{McCSul}.
Some translation is necessary, since \cite{McCSul} works with affine Coxeter groups as groups of Euclidean motions rather than as Weyl groups of affine Kac-Moody root systems.
For example, $\RSTfin{c}$ appears in \cite{McCSul} as the ``horizontal root system.''
See \cite[Definition~6.1]{McCSul} and compare the third column of \cref{tab:type-by-type} with \cite[Table~1]{McCSul}.
\end{remark}

\section{Root systems and Coxeter groups}\label{sec:root_systems}
A \newword{symmetrizable (generalized) Cartan matrix} is a square integer matrix $A=[a_{ij}]_{1\le i,j\le n}$ \nomenclature[a]{$A$}{$(a_{ij})$, Cartan matrix} with diagonal entries $2$ and nonpositive off-diagonal entries, such that there exist positive real numbers $d_i$ \nomenclature[di]{$d_i$}{(skew-)symmetrizing factors} with $d_i a_{ij}=d_j a_{ji}$ for all $i,j$.

Let $V$ \nomenclature[v]{$V$}{ambient vector space} be a complex vector space with basis $\Simples=\set{\alpha_1,\ldots,\alpha_n}$.\nomenclature[zzq]{$\Simples$}{simple roots}
We only care about the real span of $\Simples$, except when we consider eigenvectors, so it is safe to think of $V$ as a real vector space, passing to the complexification when necessary.
The element of $\Simples$ are the \newword{simple roots}\nomenclature[zza1]{$\alpha_i$}{simple root}.
Define the \newword{simple co-roots} to be $\alpha_i\ck= d_i^{-1} \alpha_i$.  

Let $K$ be the symmetric bilinear form on $V$ defined by ${K(\alpha\ck_i, \alpha_j)=a_{ij}}$.
\nomenclature[k]{$K$}{symmetric bilinear form on $V$}
For each $i=1,\ldots,n$, the \newword{simple reflection} $s_i$ \nomenclature[si]{$s_i$}{simple reflection} is the linear map given on the basis of simple roots by $s_i(\alpha_j)=\alpha_j-K(\alpha\ck_i, \alpha_j) \alpha_i$.
On the basis of simple co-roots $s_i$ acts as $s_i(\alpha\ck_j)=\alpha\ck_j-K(\alpha\ck_j, \alpha_i) \alpha\ck_i$.
The group $W$ generated by $S=\set{s_i:i=1,\ldots,n}$ is called the \newword{Weyl group}.
Each element of $W$ is a symmetry of $K$.

The \newword{real root system} $\RSre$ is the set of vectors $w\alpha_i$ for $w\in W$ and $i=1,\ldots,n$ (called \newword{real roots}).
There is a larger set $\RS\supseteq\RSre$, \nomenclature[zzv1]{$\RS$}{root system} called the \newword{root system}, which is strictly larger than $\RSre$ if and only if $\RSre$ is infinite.
We describe the \newword{imaginary roots} (the elements of $\RS\setminus\RSre$) below in the case where $A$ is of affine type.

Each root in $\RS$ is either \newword{positive} (in the nonnegative linear span of $\Simples$) or \newword{negative} (in the nonpositive linear span of $\Simples$).
Each real root $\beta$ has an associated co-root $\beta\ck=\frac{2}{K(\beta,\beta)}\beta$ and defines a \newword{reflection} $t_\beta$ on $V$ given by $t_\beta x=x-K(\beta\ck, x) \beta$ for every $x\in V$.
Every reflection in $W$ is $t_\beta$ for a unique positive real root~$\beta$. 
The notation $[\beta:\alpha_i]$ \nomenclature[zzzz]{$[\beta:\alpha_i]$}{$i$-th coefficient of $\beta$ in the basis $\Simples$} means the $\alpha_i$-coefficient of $\beta$ in the basis of simple roots.

A Cartan matrix $A$ is called \newword{reducible} if it can be reindexed to have a nontrivial block-diagonal decomposition, in which case each diagonal block is a Cartan matrix (a component of $A$).
Otherwise, it is \newword{irreducible}.
The root system $\RS$ and Weyl group $W$ associated to $A$ are accordingly reducible or irreducible.

A \newword{Coxeter element} $c$ \nomenclature[c]{$c$}{Coxeter element $s_1\cdots s_n$} is the product of any permutation of $S$.
It is possible for different permutations of $S$ to have the same product $c$.
An element $s\in S$ is \newword{initial in $c$} if there exists a permutation of $S$ whose product is $c$ and whose first entry is $s$.
Similarly, $s$ is \newword{final in $c$} if $c$ is the product of a permutation ending in $s$.
When $s$ is initial or final in $c$, the element $scs$ is also a Coxeter element for $W$.
The operation of passing from $c$ to $scs$ is called a \newword{source-sink move}.
As in the introduction, we assume that $A$ has been indexed so that $c=s_1\dots s_n$.

We use $c$ to define a skew-symmetric bilinear form on $V$ by
\begin{equation}\label{omega def}
\omega_c(\alpha_i\ck,\alpha_j)=
\left\lbrace\begin{array}{ll}
a_{ij}&\text{if }i>j,\\
0&\text{if }i=j,\text{ or}\\
-a_{ij}&i<j.
\end{array}\right.
\end{equation}
The following lemma is \cite[Lemma~3.8]{typefree}.

\begin{lemma}\label{omega invariant} 
If~$s$ is initial or final in $c$, then $\omega_c(\alpha,\beta)=\omega_{scs}(s\alpha,s\beta)$ for all $\alpha$ and~$\beta$ in $V$.
\end{lemma}

If $w\in W$ is the product $r_1\cdots r_k$ with each $r_i\in S$, then the expression $r_1\cdots r_k$ is called \newword{reduced} if every other expression $w=p_1\cdots p_\ell$ with each $p_i\in S$ has $\ell\ge k$.
An expression $r_1\cdots r_k$ is reduced if and only if its \newword{left reflections} $t_i=r_1\cdots r_i\cdots r_1$ are all distinct.
When $r_1\cdots r_k$ is reduced, the set of left reflections depends only on $w$ and is called the set of \newword{inversions} of $w$.
Furthermore, writing $\gamma_i\in\Simples$ for the simple root associated to $r_i$, the expression $r_1\cdots r_k$ is reduced if and only if the roots $r_1\cdots r_{i-1}\gamma_i$ are distinct and are all positive.
When $r_1\cdots r_k$ is reduced, these are exactly the positive roots associated to the inversions of $w$.

The following result is due to \cite{Kleiner-Pelley} in broad generality and to \cite{Speyer} in full generality.
(See also \cite{Fomin07,KMOTU}.)
Write $(s_1\cdots s_n)^k$ for the $k$-fold concatenation of $s_1\cdots s_n$ with itself.

\begin{theorem}\label{Speyer}
If $W$ is infinite and irreducible, then the word $(s_1\cdots s_n)^k$ is reduced for all $k\ge1$.
\end{theorem}

The Weyl group $W$ is finite if and only if $K$ is positive definite.
We gather some results about the finite case.
The next proposition follows from \cite[Lemma~3.16]{Humphreys}.
\begin{proposition}\label{cox fix}
If $W$ is finite and $c$ is a Coxeter element of $W$, then the fixed space of the action of $c$ on $V$ is $\set{0}$.
\end{proposition}

The following is a result of Brady and Watt \cite{BWKPi,BWOrth}.
\begin{proposition}\label{reduced T}
Suppose $W$ is a \emph{finite} Coxeter group, $w$ is an element of $W$ and $t_1\cdots t_k$ is an expression for $w$ as a product of reflections, minimizing $k$ among all such expressions.
Then $\set{\gamma_1,\ldots,\gamma_k}$ is a basis for the orthogonal complement of the fixed space of the action of $w$ on $V$, where $\gamma_i$ is the positive root such that~$t_i=t_{\gamma_i}$.
\end{proposition}
\cref{reduced T} can be obtained by combining \cite[Theorem~2.2(i)]{rotate} (another restatement of results of \cite{BWKPi,BWOrth}), which says that the fixed space of $w$ is $\bigcap_{i=1}^k \gamma_i^\perp$, with Carter's Lemma, which states that $\set{\gamma_1,\ldots,\gamma_k}$ is linearly independent. 
(See \cite{Carter} and also \cite[Lemma~2.4.5]{Armstrong}.)
The orthogonal complement the proposition refers to is with respect to the Euclidean form preserved by $W$.

A \newword{(standard) parabolic subgroup} of a Coxeter group $W$ is a subgroup generated by some subset of $S$.
The following fact is one direction of \cite[Lemma~1.4.3]{Bessis}.
\begin{lemma}\label{bessis lemma}
Suppose $c$ is a Coxeter element in a \emph{finite} Coxeter group $W$ of rank~$n$. 
If $t_1\cdots t_n$ is an expression for $c$ as a product of reflections in $W$, then for each $i$, there exists a standard parabolic subgroup $W'$ of $W$ of rank $i$ such that $t_1\cdots t_i$ is conjugate in $W$ to a Coxeter element of $W'$.
\end{lemma}

We say that $A$, $\RS$ and $W$ are of \newword{affine type} if $K$ is positive semidefinite and not positive definite and if the restriction of $K$ to $\Span\set{\alpha_i:i\in J}$ is positive definite for all $J \subsetneq [1,n]$.
In this case, $A$, $W$ and $\RS$ are in particular irreducible.
Background on affine root systems can be found in \cite{Kac90,Macdonald}.
We continue to let $n$ be the rank of $\RS$, even when $\RS$ is affine (despite a common convention where affine root systems have rank $n+1$).

If $A$ is of affine type, then there exists $\aff\in\set{1,\ldots,n}$ \nomenclature[aff]{$\aff$}{index of the affine simple root} such that, writing $W_\fin$ \nomenclature[wfin]{$W_\fin$}{finite Weyl group associated to an affine Cartan matrix} for the subgroup of $W$ generated by $S\setminus\set{s_\aff}$, the group $W$ is isomorphic to a semidirect product of $W_\fin$ with the lattice generated by $\set{\alpha_i\ck:i\neq \aff}$.
The choice of $\aff$ may not be unique, but we fix a choice.

We write $V_\fin$ \nomenclature[vfin]{$V_\fin$}{subspace spanned by $\Simples_\fin$} for the subspace of $V$ spanned by $\Simples\setminus\set{\alpha_\aff}$.
We write $\RSfin$ \nomenclature[zzv5fin]{$\RSfin$}{finite root system associated to an affine Cartan matrix} for $\RS\cap V_\fin$.
This is an indecomposable finite root system.

Some affine root systems arise from indecomposible finite root systems through a standard construction (e.g.\ \cite[Proposition~2.1]{Macdonald}).
These are the \newword{standard} affine root system, shown in Table Aff~1 of \cite[Chapter~4]{Kac90} and in Table~\ref{tab:type-by-type} of the present paper.
Every affine root system $\RS$ is a \newword{rescaling} of a unique standard affine root system $\RS'$.
This means that every root of $\RS$ is a positive scaling of a root in $\RS'$ and that the bilinear form $K$ associated to $\RS$ and $\RS'$ coincide.
In particular, both root systems define the same Weyl group, and the scaling factors relating roots are constant on $W$-orbits of roots.

When $\RS$ is of affine type, the kernel of $K$ is one-dimensional.
The intersection of~$\RS$ with this kernel is the set of imaginary roots and is of the form $\set{x\delta:x\in\integers\setminus\set{0}}$, where $\delta$ is a positive imaginary root. \nomenclature[zzd]{$\delta$}{positive imaginary root closest to the origin}
Because $\delta$ is in the kernel of $K$, it is fixed by every element of~$W$.

Every real root in $\RS$ is a positive scaling of $\beta+k\delta$ for some $\beta\in\RSfin$ and $k\in\integers$.
In some (nonstandard) affine root systems and for some roots $\beta\in\RSfin$ and integers~$k$, $\beta+k\delta$ is not a positive scaling of a root in $\RS$.
In the standard affine root systems, every $\beta+k\delta$ \emph{is} a root.
A root that is a positive rescaling of $\beta+k\delta$ is a positive root if and only if either $k$ is positive or $k=0$ and $\beta$ is positive.
See \cite[Proposition~6.3]{Kac90}.

The expansion of the imaginary root $\delta$ in the basis of simple roots has strictly positive coordinates.
In particular, $[\delta:\alpha_\aff]$ is positive.
(To be precise $[\delta:\alpha_\aff]=1$ in all affine types except for $A_{2k}^{(2)}$ where it is $2$.)
The vector $\theta=\delta-[\delta:\alpha_\aff]\alpha_\aff$ \nomenclature[zzh]{$\theta$}{highest root or highest short root in $\RSfin$} is a positive root in $\RSfin$.
(Usually, including in the standard affine root systems, $\theta$ is the highest root of $\RSfin$, but in some affine root systems, it is the highest short root \cite[Proposition~6.4]{Kac90}.)
In any case, $\alpha_\aff$ is a positive scaling of $\delta-\theta$.
Since $K(\delta,x)=0$ for any $x\in V$, the action of $s_\aff$ on $x$ is:
\begin{equation}\label{simpref}
s_\aff x=t_\theta x+K(\theta\ck,x)\delta.  
\end{equation}

To conclude the section, we quote a result on conjugacy classes of affine Coxeter elements.
In \cite[4.1]{SteinbergFRG}, it is shown that all Coxeter elements in a finite Coxeter group are conjugate, via source-sink moves.
The argument given there applies to all affine types except for $A^{(1)}_{n-1}$, where there are multiple conjugacy classes.
The classification of conjugacy classes in type $A^{(1)}_{n-1}$ is well known, and we quote it as part of the following proposition.

\begin{proposition}\label{aff conj}
If $W$ is an affine Weyl group not of type $A_{n-1}^{(1)}$, then any two Coxeter elements of $W$ are conjugate in $W$.
If $W$ is of type $A_{n-1}^{(1)}$, then there is one conjugacy class for each $k\in[1,n-1]$, represented by the Coxeter element $s_1\cdots s_n$ with $(s_1s_{k+1})^3=(s_ks_n)^3=1$ and $(s_is_{i+1})^3=1$ for $i\neq k,n$.
The conjugations can be carried out by a sequence of source-sink moves.
\end{proposition}

\section{Eigenvalues and eigenvectors of affine Coxeter elements}\label{sec:eigen}
Let $c=s_1\cdots s_n$ be a Coxeter element.
As in the introduction, let $c_\laff=s_1\cdots s_{\aff-1}$ \nomenclature[czzzzt1]{$c_\laff$}{$cs_1\cdots s_{\aff-1}$} and let $c_\gaff=s_{\aff+1}\cdots s_n$ \nomenclature[czzzzt2]{$c_\gaff$}{$s_{\aff+1}\cdots s_n$}.
The elements $c_\laff$ and $c_\gaff$ depend on how $c$ is expressed as the product of a permutation of $S$ (for example if $s_{\aff-1}$ commutes with $s_\aff$) but the statements and arguments are correct regardless of the choice.
We write $c=c_\laff s_\aff c_\gaff $ to refer to this definition of $c_\laff $ and $c_\gaff $.

\begin{proposition}\label{eigen} 
Let $c=c_\laff s_\aff c_\gaff $ be a Coxeter element in an affine Weyl group.
\begin{enumerate}[\qquad \upshape (1)]
\item \label{eigen 1}
$c$ has eigenvalue $1$ with algebraic multiplicity $2$ and geometric multiplicity~$1$.
The imaginary root $\delta$ is a $1$-eigenvector.
\item \label{eigen 1 gen}
There exists a unique generalized $1$-eigenvector $\gamma_c$ \nomenclature[zzcc]{$\gamma_c$}{$1$-eigenvector of $c$}contained in the subspace~$V_\fin$ of $V$.
(This means that $(c-1)\gamma_c=\delta$.)
The vector $\gamma_c$ is the unique vector in $V_\fin$ fixed by $c_\laff t_\theta c_\gaff $ and having $K(\theta\ck,c_\gaff \gamma_c)=1$.
\item \label{eigen neq 1}
There is a linearly independent set of $n-2$ eigenvectors of $c$ with eigenvalues $\lambda$ roots of unity but $\lambda\neq 1$.
\item \label{eigen orth}
The direct sum $\eigenspace{c}$ \nomenclature[uc]{$\eigenspace{c}$}{subspace generated by eigenvectors of $c$.  See also \cref{fin conv}}of all eigenspaces of $c$ is $\set{v\in V:K(\gamma_c,v)=0}$.
\item \label{eigen neq 1 details}
For $\lambda\neq1$, a vector $v\in V_\fin$ is a $\lambda$-eigenvector of $c_\laff t_\theta c_\gaff $ if and only if~${v+\frac{K(\theta\ck,c_\gaff v)}{(\lambda-1)}\delta}$ is a $\lambda$-eigenvector of $c$.
\end{enumerate}
\end{proposition}
\begin{proof}
Applying $c_\laff $ to the left of both sides of \cref{simpref}, keeping in mind that $\delta$ is fixed by the action of $W$, and setting $x=c_\gaff v$ for $v\in V_\fin$, we obtain 
\begin{equation}\label{simpref plus}
cv=c_\laff t_\theta c_\gaff v+K(\theta\ck,c_\gaff v)\delta
\end{equation}
Thus $c_\laff t_\theta c_\gaff v=\lambda v$ if and only if $cv-K(\theta\ck,c_\gaff v)\delta=\lambda v$.
Adding the vector $\frac\lambda{\lambda-1}K(\theta\ck,c_\gaff v)\delta$ to both sides of the latter and again keeping in mind that $\delta$ is fixed by the action of $W$, we obtain  $c\bigl(v+\frac{K(\theta\ck,c_\gaff v)}{(\lambda-1)}\delta\bigr)=\lambda\bigl(v+\frac{K(\theta\ck,c_\gaff v)}{(\lambda-1)}\delta\bigr)$.
We have established \eqref{eigen neq 1 details}.

We next check that the fixed space of $c_\laff t_\theta c_\gaff $ in $V_\fin$ is at least one-dimensional.
Since $c_\laff t_\theta c_\gaff $ is in the finite Coxeter goup $W_\fin$, and since $K$ restricts to a Euclidean form on $V_\fin$ we can use \cref{reduced T}.
Therefore to establish the claim (since the rank of $W_\fin$ is $n-1$) it suffices to check that a shortest expression for $c_\laff t_\theta c_\gaff $ as a product of reflections in $W_\fin$ has length at most $n-2$.
The same proposition implies that a shortest expression for $c_\laff t_\theta c_\gaff $ has length at most $n-1$.
But since the determinant of a reflection is $-1$, every expression for $c_\laff t_\theta c_\gaff $ as product of reflections has the same length modulo 2.
The expression $c_\laff t_\theta c_\gaff $ is a product of $n$ reflections in $W_\fin$ ($n-1$ simple reflections and $t_\theta$) and  we conclude  that a shortest expression for $c_\laff t_\theta c_\gaff $ has length at most $n-2$ as desired.

Given a vector $v\in V_\fin$, we compute $c_\laff t_\theta c_\gaff v=c_\laff (c_\gaff v-K(\theta\ck,c_\gaff v)\theta)$, so $v$ is fixed by $c_\laff t_\theta c_\gaff $ if and only if $c_\laff (c_\gaff v-K(\theta\ck,c_\gaff v)\theta)=v$ if and only if $(c_\laff c_\gaff -1) v=K(\theta\ck,c_\gaff v) c_\laff \theta$.  
But $c_\laff c_\gaff $ is a Coxeter element of $W_\fin$ and thus \cref{cox fix} says that the fixed space of $c_\laff c_\gaff $ on $V_\fin$ is $\set{0}$.  
Thus $(c_\laff c_\gaff -1)$ is invertible on $V_\fin$, so $v$ is fixed by $c_\laff t_\theta c_\gaff $ if and only if $v=K(\theta\ck,c_\gaff v)(c_\laff c_\gaff -1)^{-1}c_\laff \theta$.
In particular, the fixed space of $c_\laff t_\theta c_\gaff $ is contained in $\mathbb{R}(c_\laff c_\gaff -1)^{-1}c_\laff \theta$, but since the fixed space is at least one dimensional, it is exactly the line $\mathbb{R}(c_\laff c_\gaff -1)^{-1}c_\laff \theta$.
Furthermore, each nonzero vector $v$ fixed by $c_\laff t_\theta c_\gaff $ has $K(\theta\ck,c_\gaff v)\neq 0$.

The element $c_\laff t_\theta c_\gaff $ is an orthogonal transformation of the Euclidean vector space $V_\fin$, so it has eigenvalues which are roots of unity and also has a basis of eigenvectors.
We have seen that only one of these eigenvectors has eigenvalue $1$.
Since $\delta$ is not in $V_\fin$ and in light of \eqref{eigen neq 1 details}, the map $v\mapsto  v+\frac{K(\theta\ck,c_\gaff v)}{(\lambda-1)}\delta$ takes the $n-2$ non-fixed eigenvectors in a basis of eigenvectors of $c_\laff t_\theta c_\gaff $ to a linearly independent set of eigenvectors of $c$.
We have proved~\eqref{eigen neq 1}.

Furthermore, we see that $c$ has the eigenvalue $1$ with algebraic multiplicity $2$. 
The vector $\delta$ is a $1$-eigenvector, because it is fixed by the action of $c$.
A nonzero vector $v\in V_\fin$ fixed by $c_\laff t_\theta c_\gaff $ has $K(\theta\ck,c_\gaff v)\neq0$, so in particular, there indeed exists a unique vector $\gamma_c\in V_\fin$ fixed by $c_\laff t_\theta c_\gaff $ and having $K(\theta\ck,c_\gaff \gamma_c)=1$.
By \cref{simpref} and because $\delta$ is fixed by $W$, we compute $(c-1)\gamma_c={c_\laff t_\theta c_\gaff \gamma_c +K(\theta\ck,c_\gaff \gamma_c)\delta-\gamma_c}=\delta$.
We conclude that the the eigenvalue $1$ has geometric multiplicity $1$ and that $\gamma_c$ is a generalized $1$-eigenvector.
We have proved \eqref{eigen 1}.
A generalized $1$-eigenvector is unique up to adding a multiple of the $1$-eigenvector $\delta$, and \eqref{eigen 1 gen} follows.

By \eqref{eigen 1} and \eqref{eigen neq 1}, $c$ has a linearly independent set of $n-1$ eigenvectors.
Thus to prove \eqref{eigen orth}, it is enough to show that if $v$ is an eigenvector of $c$, then $K(\gamma_c,v)=0$. 
If $v$ has eigenvalue $\lambda$, then $K(\gamma_c,v)=K(c\gamma_c,cv)=K(\gamma_c+\delta,\lambda v)=\lambda K(\gamma_c,v)$.
If $\lambda\neq1$, then we conclude that $K(\gamma_c,v)=0$.
If $\lambda=1$, then $v$ is a multiple of $\delta$, so again $K(\gamma_c,v)=0$.
\end{proof}

We now describe how $\gamma_c$ transforms under source-sink moves.
\begin{proposition}\label{scs gamma}
Let $\RS$ be an affine root system and let $c$ be a Coxeter element.
If $s$ is initial or final in $c$, then 
\[\gamma_{scs}=\begin{cases}
s\gamma_c&\text{if }s\neq s_\aff\\
t_\theta\gamma_c&\text{if }s=s_\aff
\end{cases}\]
\end{proposition}
\begin{proof}
By \cref{eigen}\eqref{eigen 1 gen}, $\gamma_{scs}$ is the unique generalized $1$-eigenvector of $scs$ contained in $V_\fin$.
Since $\gamma_c$ is a generalized $1$-eigenvector of $c$ and $\delta$ is fixed by $s$, we have $(scs-1)(s\gamma_c)=s(c-1)\gamma_c=s\delta=\delta$, so $s\gamma_c$ is a generalized $1$-eigenvector for $scs$.
Now $\gamma_c$ is in $V_\fin$, so $s\gamma_c$ is in $V_\fin$ if $s\neq s_\aff$.
If $s=s_\aff$, then \eqref{simpref} says that $t_\theta\gamma_c$ differs from $s\gamma_c$ by a multiple of $\delta$, so $t_\theta\gamma_c$ is also a generalized $1$-eigenvector for $scs$.
Also, $t_\theta\gamma_c$ is in $V_\fin$.
\end{proof}
Combining this result with \eqref{simpref} and \cref{eigen}(\ref{eigen orth}) we get the following.
\begin{corollary}
  If $s$ is initial or final in $c$, then $\eigenspace{scs}=s\eigenspace{c}$.
\end{corollary}

Continuing on the topic of eigenvectors of $c$, we identify and characterize the $1$-eigenspace of the action of $c$ on the space dual to $V$.
We write $V^*$ \nomenclature[v_zzzzd]{$V^*$}{dual space to $V$} for the dual space to $V$ and $\br{\,\cdot\,,\,\cdot\,}$ \nomenclature[zzzz]{$\br{\,\cdot\,,\,\cdot\,}$}{canonical pairing between $V^*$ and $V$} for the canonical pairing between $V^*$ and $V$. 
As usual, the action of $W$ on $V^*$, dual to its action on $V$, is given by $\br{w\phi,v}:=\br{\phi,w^{-1}v}$.
Denote by $\phi_c$ \nomenclature[zzv8c]{$\phi_c$}{$K(\gamma_c,\,\cdot\,)$} the element of $V^*$ defined by $\br{\phi_c,v}=K(\gamma_c,v)$ for all $v\in V$.
\begin{lemma}\label{phic lemma}
  The fixed space of the action of $c$ on $V^*$ is $\reals\phi_c$.
\end{lemma}
\begin{proof}
We make a straightforward computation, using the definition of $\phi_c$, the fact that $K$ is invariant under the action of $W$, \cref{eigen}\eqref{eigen 1 gen}, and the fact that $K(\delta,v)=0$ for all $v\in V$. 
Specifically, for $v$ in $V$, we calculate
		\[
			\br{c\phi_c,v}=
			\br{\phi_c,c^{-1}v}=
			K(\gamma_c,c^{-1}v)=
			K(c\gamma_c,v)=
			K(\delta + \gamma_c,v)=
			K(\gamma_c,v)=
			\br{\phi_c,v}.
		\]
The fixed space of $c$ acting on $V^*$ is a line because the fixed space of $c$ on $V$ is a line.
\cref{eigen}\eqref{eigen orth} implies that $\phi_c\neq0$, so it spans the fixed space.
\end{proof}

We now relate $\phi_c$ to a vector in $V^*$ that plays a key role in \cite{afframe}.
Following \cite[Section~4.4]{afframe}, we write $x_c$ \nomenclature[xc]{$x_c$}{$\omega_c(\delta,\,\cdot\,)$} for the linear functional $\omega_c(\delta,\,\cdot\,)$ in $V^*$.  

\begin{lemma}\label{xc lemma}
The vector $x_c$ is fixed by $c$ and is a negative scaling of $\phi_c$.
\end{lemma}
\begin{proof}
Using first the definition of $x_c$, then \cref{omega invariant} ($n$ times), then the fact that~$\delta$ is fixed by the action of $W$, for any vector $v$ in $V$, we calculate
		\[
			\br{cx_c,v}=
			\br{x_c,c^{-1}v}=
			\omega_c(\delta,c^{-1}v)=
			\omega_c(c\delta,v)=
			\omega_c(\delta,v)=
			\br{x_c,v}. 
		\]
Thus $x_c$ is fixed by the action of $c$ on $V^*$.
By \cref{phic lemma}, $x_c$ and $\phi_c$ agree up to scaling.
To determine the scale factor, we again write $c=c_\laff s_\aff c_\gaff$ and recall from the definition of $\gamma_c$ that $K(\theta\ck,c_\gaff\gamma_c)=1$.
Using the invariance of $K$, we have $K(c_\gaff^{-1}\theta\ck,\gamma_c)=1$ or in other words $\br{\phi_c,c_\gaff^{-1}\theta\ck}=1$.
Again using \cref{omega invariant} several times, then the $W$-invariance of $\delta$, then the fact that $\omega_c$ is skew-symmetric, we calculate $\br{x_c,c_\gaff^{-1}\theta\ck}$ to be
\[\omega_c(\delta,c_\gaff ^{-1}\theta\ck)=\omega_{c_\gaff c_\laff s_\aff}(\delta,\theta\ck)=\omega_{c_\gaff c_\laff s_\aff}(f\alpha_\aff+\theta,\theta\ck)=f\omega_{c_\gaff c_\laff s_\aff}(\alpha_\aff,\theta\ck),\]
with $f=1$ except in type $A^{(2)}_{2k}$, where $f=2$.
Now $\omega_{c_\gaff c_\laff s_\aff}(\alpha_\aff,\theta\ck)$ equals $-\omega_{c_\gaff c_\laff s_\aff}(\theta\ck,\alpha_\aff)$, which by \cref{omega def} is negative.
\end{proof}

We make note of a formula for $x_c$ that is obtained by using \cref{omega def} to evaluate $x_c$ on each simple co-root.
We write $\set{\rho_i:i=1,\ldots,n}$ for the \newword{fundamental weights}, the basis of $V^*$ that is dual to the basis $\Simples\ck$ of co-roots.
\begin{lemma}\label{phic formula}
$\displaystyle x_c=\sum_{1\le i<j\le n}\bigl([\delta:\alpha_j]a_{ij}\rho_i-[\delta:\alpha_i]a_{ji}\rho_j\bigr)$.
\end{lemma}

\section{$c$-Orbits of roots}\label{sec:orbits} 
In this section, we prove \cref{aff c-orbits}.
We begin with a technical lemma.

\begin{lemma}[Cf. \protect{\cite[Lemma~1.6]{Dlab76}}]\label{lemma:change}
Suppose $\beta$ is a positive root in $\RS$. 
Then $c\beta$ is negative if and only if $\beta\in\TravInj{c}$, in which case $-c\beta\in\TravProj{c}$.
Also $c^{-1}\beta$ is negative if and only if $\beta\in\TravProj{c}$, in which case $-c^{-1}\beta\in\TravInj{c}$.
\end{lemma}
\begin{proof}
A root $\beta$ changes sign under the action of a simple reflection $s_i$ if and only if $\beta\in\set{\pm\alpha_i}$.
Thus $c\beta$ is negative if and only if there exists an index $i\in\set{1,\ldots,n}$ such that $c=c_1s_ic_2$ and $c_2\beta=\alpha_i$, so that $\beta=c_2^{-1}\alpha_i\in\TravInj{c}$. 
In this case, $-c\beta=-c_1(-\alpha_i)\in\TravProj{c}$.
Similarly, $c^{-1}\beta$ is negative if and only if there exists $i\in\set{1,\ldots,n}$ with $c=c_1s_ic_2$ and $c_1^{-1}\beta=\alpha_i$, so that $\beta=c_1\alpha_i\in\TravProj{c}$.
In this case, $-c^{-1}\beta=-c_2^{-1}(-\alpha_i)\in\TravInj{c}$.
\end{proof}	

\cref{lemma:change} leads to the following proposition.
\begin{proposition}[Cf. \protect{\cite[Lemma~1.7]{Dlab76}}]\label{Uc compl}
If $\beta\in\RS\setminus \eigenspace{c}$ then there exists $k\in\mathbb{Z}$ such that $c^k\beta\in\TravInf{c}$.
\end{proposition}
\begin{proof}
We will show that $\set{c^k\beta:k\in\integers}$ contains both positive and negative roots.
Then \cref{lemma:change} completes the proof.

We write $\beta=a\gamma_c+b\delta+v$ with $v\in \eigenspace{c}\cap V_\fin$.
Since $\beta\in\RS\setminus \eigenspace{c}$, we have $a\neq 0$.
By \cref{eigen}\eqref{eigen neq 1 details}, $c$ acts on $\eigenspace{c}$ with the same eigenvalues as $c_\laff t_\theta c_\gaff$. 
The latter is of finite order, so there is a positive integer $\ell$ such that $c^\ell v=v$. 
Using \cref{eigen}\eqref{eigen 1 gen} to compute $c\gamma_c$, we see that $c^{m\ell}\beta$ is $a\gamma_c+(b+m\ell a)\delta+v$ for any~$m\in\integers$.

Since $a\gamma_c+v\in V_\fin$, we have $[a\gamma_c+v:\alpha_\aff]=0$.  
Thus ${[c^{m\ell}\beta:\alpha_\aff]}$ equals ${(m\ell a+b)[\delta:\alpha_\aff]}$. 
Since $\ell a\neq 0$, by varying $m$, the expression ${(m\ell a+b)[\delta:\alpha_\aff]}$ can be made positive or negative.
\end{proof}

\begin{proof}[Proof of \cref{aff c-orbits}(\ref{criterion})]
  The proof of \cref{Uc compl} also establishes that for $\beta\in\RS\setminus \eigenspace{c}$, the set $\set{[c^m\beta:\alpha_\aff]:m\in\integers}$ is infinite and therefore the orbit of $\beta$ is infinite.
On the other hand, \cref{eigen}\eqref{eigen neq 1 details} implies that $c$ is diagonalizable on $\eigenspace{c}$ with eigenvectors that are roots of unity.
Thus $c$ has finite order on $\eigenspace{c}$.
\end{proof}

\cref{Uc compl} is also a step in the direction of proving \cref{aff c-orbits}\eqref{infinite transversal}.
To complete the proof, we need to show that the elements of $\TravInf{c}$ are in distinct infinite orbits.
We will see that this assertion is essentially equivalent to \cref{Speyer}.

\begin{proof}[Proof of \cref{aff c-orbits}(\ref{infinite transversal})] 
For any $k$, index the simple roots modulo $n$ to write the word $(s_1\cdots s_n)^k$ from \cref{Speyer} as $s_1s_2\cdots s_{kn}$.
Consider the $kn$ roots $\alpha_1$, $s_1\alpha_2$, $s_1s_2\alpha_3$, etc.
The conclusion of \cref{Speyer} is that these roots are all distinct and positive.
But these roots are exactly the roots $\TravProj{c}\cup c\TravProj{c}\cup\cdots\cup c^{k-1}\TravProj{c}$, which is therefore a disjoint union of $k$ sets of size $n$.
This is true for any $k\ge0$, and we conclude: first, that the $c$-orbit of any element of $\TravProj{c}$ is infinite; second, that $c^m\alpha$ is a positive root for any $\alpha\in\TravProj{c}$ and $m\ge0$; and third, that there do not exist $\alpha,\beta\in\TravProj{c}$ and $m\ge0$ such that $\beta=c^m\alpha$, unless $\alpha=\beta$ and $m=0$.
Since $\TravInj{c}=\TravProj{c^{-1}}$, we can also conclude that the $c$-orbit of any element of $\TravInj{c}$ is infinite, that $c^m\alpha$ is positive for $\alpha\in\TravInj{c}$ and $m\le0$, and that there do not exist $\alpha,\beta\in\TravInj{c}$ and $m\le0$ such that $\beta=c^m\alpha$, unless $\alpha=\beta$ and $m=0$.
It remains only to rule out the possibility that there exist $\alpha\in\TravProj{c}$ and $\beta\in\TravInj{c}$ such that $\beta=c^m\alpha$ for some integer $m$.
But in this case, if $m\ge 0$, then $c^{m+1}\alpha$ is a negative root by \cref{lemma:change}, contradicting the fact established above that $c^{m+1}\alpha$ is positive.
If $m<0$, then we reach a similar contradiction because $\alpha\in\TravInj{c^{-1}}$, $\beta\in\TravProj{c^{-1}}$ and $\alpha=c^{-m}\beta$.
\end{proof}

We pause in the proof of \cref{aff c-orbits} to make a useful observation that follows from the proof of \cref{aff c-orbits}(\ref{infinite transversal}).

\begin{proposition}[Cf. \protect{\cite[Proposition~1.9]{Dlab76}}]\label{useful}
The orbits of roots in $\TravProj{c}$ are separated from the orbits of roots in $\TravInj{c}$ by the hyperplane $\eigenspace{c}$.
Specifically, $K(\gamma_c,\beta)>0$ for $\beta\in c^{m}\TravProj{c}$ and $m\in\mathbb{Z}$, while 
$K(\gamma_c,\beta)<0$ for $\beta\in c^{m}\TravInj{c}$ and $m\in\mathbb{Z}$.
\end{proposition}
\begin{proof} 
\cref{lemma:change} implies that a root $\beta$ is in the $c$-orbit of an element of $\TravInj{c}$ if and only if $-\beta$ is in the $c$-orbit of an element of $\TravProj{c}$.
  In particular it suffices to establish the claim for positive roots.
The set $\bigcup_{m\ge0}c^{m}\TravProj{c}$ is $\set{s_1\cdots s_{k-1}\alpha_k:k\ge1}$ with indices interpreted modulo~$n$.
Half of the proof is to show that $K(\gamma_c,s_1\cdots s_{k-1}\alpha_k)>0$ for all $k$.
We argue by induction, with $c$ allowed to vary.
If $k=0$, then $K(\gamma_c,\alpha_1)$ is a negative multiple of $\omega_c(\delta,\alpha_1)$ by \cref{xc lemma}.
Since $W$ is irreducible (because it is affine), we see from \eqref{omega def} that $\omega_c(\delta,\alpha_1)<0$.
If $k>0$, then by induction $K(\gamma_{s_1cs_1},s_2\cdots s_{k-1}\alpha_k)>0$.
By \cref{scs gamma} and \eqref{simpref}, $\gamma_c$ differs from $s_1\gamma_{s_1cs_1}$ by a multiple of $\delta$.
Since $K$ is $W$-invariant and since $K(\delta,\,\cdot\,)=0$, we see that $K(\gamma_c,s_1\cdots s_{k-1}\alpha_k)>0$.
The other half of the proof is symmetric, switching $c$ with~$c^{-1}$.
\end{proof}

The remaining parts of \cref{aff c-orbits} concern finite orbits.
One of them, \cref{aff c-orbits}(\ref{im orb}), is trivial because the action of $W$ fixes all imaginary roots.
Another of them is now easily proved:
\cref{aff c-orbits}\eqref{infinite transversal} combines with \cref{lemma:change} to imply that roots only change sign within infinite orbits, and \cref{aff c-orbits}(\ref{pos neg}) follows.

It remains to prove \cref{aff c-orbits}(\ref{finite transversal},\ref{finite transversal finite}).
The easiest assertion left to prove is that there are infinitely many finite $c$-orbits of real roots, or equivalently (in light of \cref{aff c-orbits}\eqref{criterion}) that there are infinitely many real roots in $\eigenspace{c}$.
We begin by showing that there exist real roots in $\eigenspace{c}$, specifically, roots in $\RSfin\cap \eigenspace{c}$.
Recall the notation $\RSTfin{c}$ for the intersection $\RSfin\cap \eigenspace{c}$.
Since $\eigenspace{c}$ is a subspace, $\RSTfin{c}$ is a (necesarily finite) root subsystem of $\RSfin$, but \textit{a priori} it may be empty.
To the contrary, we will see that it is as large as it could be given the dimension of $\eigenspace{c}\cap V_\fin$.
Understanding the structure of $\RSTfin{c}$ will turn out to be critically important.

\begin{proposition}\label{Up details}
Let $\RSTfin{c}$ be the finite root system $\RSfin\cap \eigenspace{c}$ and write $c=c_\laff s_\aff c_\gaff$.
\begin{enumerate}[\qquad \upshape (1)]
\item \label{Upsilon full rank}
$\RSTfin{c}$ has rank $n-2$.
\item \label{type A}
The irreducible components  of $\RSTfin{c}$ are all of finite type $A$.  
\item \label{canon order}
The roots in the symple system $\SimplesT{c}_\fin$ of $\RSTfin{c}$ can be ordered $\beta_1,\ldots,\beta_{n-2}$ so that $t_{\beta_1}\cdots t_{\beta_{n-2}}=c_\laff t_\theta c_\gaff$.
\item \label{Omega indep}
The order $\beta_1,\ldots,\beta_{n-2}$ of \eqref{canon order} may not be unique, but as long as $t_{\beta_1}\cdots t_{\beta_{n-2}}$ is $c_\laff t_\theta c_\gaff$, the set $\TravReg{c}=\set{\beta_1,t_{\beta_1}\beta_2,\ldots,t_{\beta_1}\cdots t_{\beta_{n-2}}\beta_{n-2}}$ does not depend on the order.
\end{enumerate}
\end{proposition}

\begin{proof}
We already showed (in the proof of \cref{eigen}) that the fixed space of $c_\laff t_\theta c_\gaff$ (in $V_\fin$) is spanned by $\gamma_c$.
Since there exists some minimal expression for $c_\laff t_\theta c_\gaff$ as a product of reflections in $W_\fin$, \cref{reduced T} implies that $\RSfin$ contains a basis for $\set{v\in V_\fin:K(\gamma_c,v)=0}=\eigenspace{c}\cap V_\fin$.
This space is $(n-2)$-dimensional, so the basis has $n-2$ vectors, and therefore $\RSTfin{c}$ has rank $n-2$, and we have proven~\eqref{Upsilon full rank}.

Furthermore, when we write the expression $t_1\cdots t_{n-2}$ for $c_\laff t_\theta c_\gaff$, we can define $t_{n-1}$ to be the reflection $c_\gaff^{-1}t_\theta c_\gaff$, so that $t_1\cdots t_{n-1}=c_\laff c_\gaff$.
Since $c_\laff c_\gaff$ is a Coxeter element for $W_\fin$, \cref{bessis lemma} says that $t_1\cdots t_{n-2}$ is conjugate to a Coxeter element of some parabolic subgroup of $W_\fin$.

Define $\Proj_\fin:V\to V_\fin$ to be the linear map that fixes $V_\fin$ pointwise and sends $\delta$ to zero.
By \cref{simpref plus}, we conclude that $c_\laff t_\theta c_\gaff$ and $\Proj_\fin\circ\,c$ coincide, as maps on~$V_\fin$.

Now consider any irreducible component $\Theta$ of $\RSTfin{c}$ and any expression $t_1\cdots t_{n-2}$ for $c_\laff t_\theta c_\gaff$ as a product of $n-2$ reflections in $W_\fin$.
Since the reflections $t_i$ all correspond to roots in $\RSTfin{c}$ and since $\Theta$ is an irreducible component, we can, by transposing commuting reflections if necessary, assume that the reflections corresponding to roots in $\Theta$ are $t_1,\ldots,t_k$ for some $k$.
Since $t_1\cdots t_{n-2}$ is conjugate to a Coxeter element of some (possibly reducible) parabolic subgroup of $W_\fin$ the same conjugation takes~$t_1\cdots t_k$ to a Coxeter element of a (typically smaller and always irreducible) parabolic subgroup.
The elements $t_1\cdots t_k$, $c_\laff t_\theta c_\gaff$, and $\Proj_\fin\circ\,c$ all have the same action on $\Theta$.

Let $\widetilde\Theta$ be the root subsystem $\RS\cap\Span(\Theta\cup\set{\delta})$ of $\RS$.
Then $\widetilde\Theta$ is an affine root system whose associated finite root system is $\Theta$.
The action of $c$ fixes $\widetilde\Theta$ as a set and takes positive roots to positive roots, so the canonical simple system for $\widetilde\Theta$ is a union of $c$-orbits.
(If some power of $c$ maps a root in the canonical simple system to a root not in the canonical simple system, then that power of $c$ maps the canonical simple system to a different simple system defining the same positive roots, and that is a contradiction.)
These orbits are the same size as $t_1\cdots t_k$-orbits of roots in $\Theta$, and since $t_1\cdots t_k$ is conjugate to a Coxeter element, \cref{fin c-orbits} says that each orbit has size $h$, the Coxeter number of $\Theta$.
The Coxeter number is always strictly greater than the rank and the canonical simple system of $\widetilde\Theta$ consists of $k+1$ roots.
We conclude that the canonical simple system is a single $c$-orbit.

We now consider what the type of $\widetilde\Theta$ might be.
The element $c$ restricts to an automorphism of its Dynkin diagram, and this automorphism has a single orbit.
Since each affine Dynkin diagram is either a tree or a cycle (the latter with only single edges), we conclude that $\widetilde\Theta$ is of type $A^{(1)}_k$.
Thus $\Theta$ is of type $A_k$.
(Alternately, we reason based on the fact that the Coxeter number of $\Theta$ is one more than its rank.)
This completes the proof of \eqref{type A}.

Numbering the simple roots of $\Theta$ as $\beta_1$ through $\beta_k$ linearly along the Dynkin diagram, and numbering the additional simple root of $\widetilde\Theta$ as $\beta_0$, the action of $c$ on $\widetilde\Theta$ is by the diagram automorphism $i\mapsto i+1$ (mod $k+1$).
Thus $t_1\cdots t_k$ acts by $\beta_i\mapsto\beta_{i+1}$ for $i=1,\ldots k-1$ and by $\beta_k\mapsto-(\beta_1+\cdots+\beta_k)\mapsto\beta_1$.
The same action is accomplished by the Coxeter element $t_{\beta_1}\cdots t_{\beta_k}$ of $\Theta$, and we conclude that $t_1\cdots t_k=t_{\beta_1}\cdots t_{\beta_k}$.
Since $\Theta$ is an irreducible component of $\RSTfin{c}$ (and since reflections in different irreducible components commute), by induction of the number of irreducible components, we see that an appropriate numbering of the simple roots of $\RSTfin{c}$ yields $t_1\cdots t_{n-2}=t_{\beta_1}\cdots t_{\beta_{n-2}}$.
This is \eqref{canon order}.

Finally, we prove \eqref{Omega indep}.
Since $\set{\beta_i:i=1,\ldots,n-2}$ is a simple system for $\RSTfin{c}$, the element $c_\laff t_\theta c_\gaff$ is a Coxeter element for the associated Weyl group.
Each orderings on the simple system such that $t_{\beta_1}\cdots t_{\beta_{n-2}}=\Proj_\fin\circ\,c$ is a reduced word for $\Proj_\fin\circ\,c$.
The set $\TravReg{c}$ is the set of positive roots associated to the inversions of $\Proj_\fin\circ\,c$ and is thus independent of the choice of reduced word.
\end{proof}

\begin{remark}
Since the sub-root system $\RSTfin{c}$ has only type-$A$ components, there is a straightforward characterization of the roots $\{\beta_j\}_{j\in [1,n-2]}$ for any choice of Coxeter element $c$. 
These are the roots in $\set{\alpha\in \RSfinpos: K(\gamma_c,\alpha)=0}$ that are not sums of other roots in the same set.
\end{remark}

The proof of \cref{Up details} also establishes the following fact:  
\begin{proposition}\label{c on tUp components}
Each component of $\RST{c}$ is of affine type $A^{(1)}$.
The action of $c$ on each component is to rotate the Dynkin diagram of the component, taking each node to an adjacent node (or when the component has rank two, to transpose the two nodes of the Dynkin diagram).
\end{proposition}

\begin{proof}[Proof of \cref{aff c-orbits}(\ref{finite transversal},\ref{finite transversal finite})]
  If $\RS$ is of rank $2$, then $\eigenspace{c}=\reals\delta$, so \cref{aff c-orbits}\eqref{criterion} implies that there are no finite $c$-orbits of real roots.
If $\RS$ is of rank at least $3$, then $\RSTfin{c}$ has rank at least $1$, and thus has at least one irreducible component.  
In particular, $\RST{c}$ contains at least one infinite root subsystem (called $\widetilde\Theta$ in the proof of \cref{Up details}\eqref{type A}) so there are infinitely many finite $c$-orbits.

Let $\Theta$ be an irreducible component of $\RSTfin{c}$ and write $\widetilde\Theta=\RS\cap{\Span(\Theta\cup\set{\delta})}$. 
Ordering the canonical simple system of $\Theta$ as $\beta_1,\ldots,\beta_k$ as in the proof of \cref{Up details}, the intersection $\TravReg{c}\cap\Theta$ is 
\[\set{\beta_1,t_{\beta_1}\beta_2,\ldots,t_{\beta_1}\cdots t_{\beta_{k-1}}\beta_k}=\set{\beta_1,\beta_1+\beta_2,\ldots,\beta_1+\cdots+\beta_k}.\]
Write $\beta_0$ for the root $c\beta_k$.
In light of \cref{Up details}\eqref{type A} and \cref{c on tUp components}, we see that $\beta_0+\beta_1+\cdots+\beta_k$ is in the line spanned by $\delta$, and thus equals $\kappa\delta$ for some positive integer $\kappa$.
Furthermore, we see that $\kappa(\beta)=\kappa$ for every $\beta\in\TravReg{c}\cap\Theta$.

The Coxeter element $c$ acts by $\beta_i\mapsto\beta_{i+1}$ (mod $k+1$).
Since the positive roots of $\Theta$ are of the form $\beta_i+\cdots+\beta_j$ for $1\le i\le j\le k$, we see that every positive root of $\Theta$ is in the $c$-orbit of exactly one root in $\TravReg{c}\cap\Theta$.
The positive roots of $\widetilde\Theta$ are positive roots of $\Theta$ plus nonnegative multiples of $\kappa\delta$ or negative roots of $\Theta$ plus positive multiples of $\kappa\delta$.
Thus every positive root of $\widetilde\Theta$ is contained in the $c$-orbit of exactly one root of the form $\beta+m\cdot\kappa\delta$ with $\beta\in\TravReg{c}\cap\Theta$ and $m\ge0$.
Applying the antipodal map, we see that every negative root of $\widetilde\Theta$ is in the $c$-orbit of exactly one root of the form $-\beta+m\cdot\kappa\delta$ with $\beta\in\TravReg{c}\cap\Theta$ and $m\le0$.
Since $-\beta+\kappa\delta$ is in the $c$-orbit of some element of $\TravReg{c}\cap\Theta$, we see that every negative root of $\widetilde\Theta$ is in the $c$-orbit of exactly one root of the form $\beta+m\cdot\kappa(\beta)\delta$ with $\beta\in\TravReg{c}\cap\Theta$ and $m<0$.

We have proved \cref{aff c-orbits}(\ref{finite transversal},\ref{finite transversal finite}) in the case where $\RSTfin{c}$ has exactly one irreducible component.
Since roots in different irreducible components are orthogonal, the general case of (\ref{finite transversal},\ref{finite transversal finite}) follows easily.
\end{proof}


	\bibliographystyle{plain}
	\bibliography{bibliography}
  \vspace{-0.175 em}

\end{document}